\documentclass[A4,twoside,11pt,reqno]{article}
\usepackage{amssymb}
\usepackage{amsthm}
\usepackage[tbtags]{amsmath}
\usepackage{doc}
\usepackage{latexsym}
\usepackage{amscd}
\usepackage{graphics}
\usepackage{epsfig}
\usepackage{eucal}
\usepackage{mathrsfs}
\newtheorem{theorem}{Theorem}[section]
\newtheorem{proposition}[theorem]{Proposition}
\newtheorem{lemma}[theorem]{Lemma}
\newtheorem{corollary}[theorem]{Corollary}

\theoremstyle{definition}

\theoremstyle{remark}

\numberwithin{equation}{section}

\begin{document}

\title{{\bf Remarks on quadratic left Bol algebras}}

\author{A. Nourou Issa \\ {\it D\'epartement de Math\'ematiques Universit\'e d'Abomey-Calavi}\\{\it 01 BP 4521 Cotonou 01, Benin}\\
{\it woraniss@yahoo.fr}}



\date{}
\maketitle

\begin{abstract}
In this paper the notion of a quadratic (left) Bol algebra is discussed. Several examples of quadratic Bol algebras are given and it is observed that the only two-dimensional quadratic real Bol algebras are quadratic Lie triple systems. Dual representations of Bol algebras are investigated with a particular emphasis on coadjoint representations for quadratic Bol algebras. The notion of $T^*$-extension of a quadratic Bol algebra is introduced.
\end{abstract}

{\it keywords:} Quadratic algebra; dual representation; $T^*$-extension.\\



\section{Introduction}

Smooth (or analytic) loops are known to be a nonassociative generalization of Lie groups whose tangent structures are Lie algebras. Therefore many mathematicians considered the problem of the suitable tangent algebra for some local smooth loop (i.e. a setting for a general nonassociative Lie theory). The first attempt in solving this problem seems to be done by M. A. Akivis \cite{Ak} who introduced the notion of a $W$-algebra (now termed Akivis algebra \cite{HS}) as a tangent algebra to a local $3$-web. A decade later, L. V. Sabinin and P. O. Mikheev \cite{SM2} introduced the class of hyperalgebras (now called Sabinin algebras) as a broad generalization of Akivis algebras, and this finally solved the problem of the extension of the Lie theory up to smooth loops. For a modern treatment of Sabinin algebras, one refers to \cite{MP, MPS, Per1, SU, Wei}.
\par
It should be observed that the very beginnings towards a generalization of Lie theory were related to some types of smooth loops such as Moufang loops with the pioneer paper by A. I. Maltsev \cite{Mal} who introduced the class of Moufang-Lie algebras (now called Maltsev algebras \cite{Sag}) as tangent algebras to local smooth Moufang loops. Generalizing the E. Cartan's study of the differential geometry of Lie groups, L. V. Sabinin and P. O. Mikheev investigated the differential geometry of smooth Bol loops giving rise to the notion of a Bol algebra as a tangent algebra to a local smooth Bol loop (\cite{SM1, Mik}). Since their introduction, papers devoted to Bol algebras are relatively scarce in contrast to the case of Lie-Yamaguti algebras (i.e. generalized Lie triple systems \cite{Yam2} or Lie triple algebras \cite{K1}). Solvability and semisimplicity as well as quadratic structures of Bol algebras were introduced and investigated by E. N. Kuzmin and O. Zaidi in \cite{KZ}. For other investigated so far topics on Bol algebras one may refer to \cite{ACF, BM, Fil, HP, I1, Per, ZZ}.
\par
For a given left Bol algebra, a representation was defined in \cite{I1}. In this paper we define a coadjoint representation for a quadratic left Bol algebra, and next its coadjoint representation is used to construct a chain of quadratic left Bol algebras. An attempt towards a study of dual representations of (not necessarily quadratic) left Bol algebras is also done. The notion of a $T^*$-extension of a quadratic left Bol algebra is introduced.
\par
The organization of the paper is as follows.  In section 2 we recall some basic notions and fundamental results on Bol algebras. Examples of quadratic left Bol algebras are given and symmetric bilinear forms induced by a quadratic structure on a Bol algebra are considered. In section 3 we study dual representations of Bol algebras, either quadratic or not. While for a quadratic Bol algebra a representation as defined in \cite{I1} naturally induces a coadjoint representation, for a general Bol algebra $T$ (i.e. not necessarily quadratic) the notion of dual representation is subjected to some conditions (we still do not know whether another specific definition of the dual representation maps could avoid resorting to these conditions). In section 4, for a given quadratic Bol algebra, its $T^*$-extension is defined as a generalization of the one for Lie triple systems \cite{LWD}.
\par
Throughout this paper all vector spaces and algebras are finite-dimensional over $\mathbb{R}$, unless otherwise stated.

\section{Definitions, Examples, and Induced Bilinear Forms}

In this section we recall some basic notions and constructions on Bol algebras including the notion of an invariant bilinear form as in \cite{KZ} for right Bol algebras. Next we give a definition of a quadratic (left) Bol algebra according to the equivalence between the left and the right invariant conditions in Lie triple systems. Several examples of quadratic Bol algebras are given. It is proved that, except two-dimensional Lie triple systems, there does not exist any quadratic two-dimensional real Bol algebra. Finally, symmetric bilinear forms induced by a quadratic structure on a given Bol algebra are investigated.\\
\\
\noindent
{\bf Definition 2.1.} (\cite{Mik, SM1}).
 A left Bol algebra is a triple $(T, *, [\![, ,]\!])$ consisting of a
 vector space $T$, a binary operation $* : T \times T \rightarrow T$ and a ternary operation $[\![ , , ]\!]: T \times T  \times T \rightarrow T$ satisfying the following identities:
 \par
 (T01) $x*y = - y*x$,
 \par
 (T02) $[\![ x,y,z ]\!] = -[\![ y,x,z ]\!]$,
 \par
 (T1) $[\![ x,y,z ]\!] +  [\![ y,z,x ]\!] + [\![ z,x,y ]\!]= 0$,
 \par
 (T2) $[\![ u,v, [\![ x,y,z ]\!] ]\!]=[\![ [\![u,v,x]\!], y,z ]\!] + [\![ x,[\![u,v,y]\!],z ]\!] + [\![ x,y,[\![u,v,z]\!] ]\!]$,
 \par
 (T3) $ [\![ u,v,x*y ]\!] = [\![ u,v,x ]\!]*y + x*[\![ u,v,y ]\!] + [\![ x,y, u*v ]\!] - xy * uv$\\
 for all $u,v,x,y,z \in T$, where $ xy * uv := (x*y)*(u*v)$.
\\
\par
The definition given above is the one of a {\it left} Bol algebra. Setting $x \cdot y := -x*y$ and $\langle z,x,y \rangle := - [\![ x,y,z ]\!]$, one gets a {\it right} Bol algebra
$(T, \cdot , \langle , , \rangle)$. In \cite{KZ} only right Bol algebras were considered. Throughout this paper by a Bol algebra we always mean a left Bol algebra.\\
\\
\noindent
{\bf Remark 2.2.}
The identities (T02), (T1), and (T2) define a {\it Lie triple system} $(T, [\![, ,]\!])$. Therefore a Bol algebra could be seen as a Lie triple system with an additional binary operation satisfying (T01) and (T3).
\\

\noindent
{\bf Example 2.3.}
In \cite{KZ} a classification of real two-dimensional right Bol algebras was given. From this classification one gets that the following table
\par
$e_1 * e_2 = -e_2$, $[\![e_1,e_2 ,e_1]\!] = e_2$, $[\![e_1,e_2 ,e_2]\!] = -e_1$\\
defines a real two-dimensional left Bol algebra $(T, *, [\![, ,]\!])$ with basis $\{ e_1, e_2\}$.
\\
\par
In fact the classification in \cite{KZ} can be reformulated for real two-dimensional left Bol algebras as follows.

\begin{theorem}\label{th2.1}
(\cite{KZ}).
Every two-dimensional left Bol algebra $(T, *, [\![, ,]\!])$ over
$\mathbb{R}$ has a canonical basis $\{ e_1, e_2\}$ such that $(T, *, [\![, ,]\!])$ is of one of the following types:
\par
${\mathrm{I.}}$ $e_1 * e_2 = 0$, $[\![e_1,e_2 ,e_1]\!] = {\varepsilon}_2 e_2$, $[\![e_1,e_2 ,e_2]\!] = - {\varepsilon}_1 e_1$,\\
where $({\varepsilon}_1 , {\varepsilon}_2) = (0,0), (\pm 1,0), (1, \pm 1)$ or
$(-1,-1)$;
\par
${\mathrm{II. (i)}}$  $e_1 * e_2 =- e_2$, $[\![e_1,e_2 ,e_1]\!] = {\beta} e_2$, $[\![e_1,e_2 ,e_2]\!] = - {\varepsilon} e_1$,\\
where $\varepsilon = 0$ or $\pm 1$;
\par
${\mathrm{(ii)}}$ $e_1 * e_2 =- e_2$,
$[\![e_1,e_2 ,e_1]\!] = e_1$, $[\![e_1,e_2 ,e_2]\!] = - e_2$.
\end{theorem}

Clearly the algebras of type I are real two-dimensional Lie triple systems. A classification of two-dimensional Lie triple systems over
$\mathbb{C}$ was given in \cite{Yam1}.\\
\\
\noindent
{\bf Example 2.4.}
 A {\it Maltsev algebra} \cite{Mal} is a pair
 $(M,*)$ consisting of a vector space $M$ and a binary operation $* : M \times M \rightarrow M$ such that
 \par
 $x*y = - y*x$,
 \par
 $xy*xz = (xy*z) *x + (yz*x)*x + (zx*x)*y$\\
 for all $x,y,z \in M$. First papers on algebraic theory of Maltsev algebras were \cite{Mal, Kl, Sag, Yam3, Yam4, Loo, Kuz} whereas semisimple Maltsev algebras were investigated in \cite{Eld, Loo, R1}. Connections between Maltsev algebras and ternary (resp. binary-ternary) algebras were found first in \cite{Loo} (resp. \cite{Yam3, Yam4}).
 \par
 If define on $(M,*)$ a ternary operation by
 \begin{equation}\label{eq2.1}
  \langle x,y,z \rangle = 2xy*z - yz*x - zx*y
 \end{equation}
$\forall x,y,z \in M$, then $(M, \langle , , \rangle)$ is a Lie triple system (\cite{Loo}) and from \cite{Mik} one knows that $(M, *, [\![, ,]\!])$ is a left Bol algebra (the Bol algebra {\it associated} with
$(M,*)$) where $[\![x,y,z]\!] := \frac{1}{3} \langle x,y,z \rangle $.
\\
\par
Given a Bol algebra $(T, *, [\![, ,]\!])$, let $b : T \times T \rightarrow \mathbb{R}$ be a symmetric bilinear form. For the triple $(T, *, [\![, ,]\!])$, we shall consider the following conditions:
\begin{equation}\label{eq2.2}
 b(x*y,z) = b(x, y*z),
\end{equation}
\begin{equation}\label{eq2.3}
 b([\![x,y,z]\!] ,u) = b([\![z,u,x]\!], y),
\end{equation}
\begin{equation}\label{eq2.4}
 b([\![x,y,z]\!] ,u) = -b(z,[\![x,y,u]\!])
\end{equation}
for all $u,x,y,z \in T$. Usually (\ref{eq2.2}) is called the {\it associative} condition on $(T,*)$ while (\ref{eq2.3}) (resp. (\ref{eq2.4})) is called the {\it right} (resp. {\it left}) {\it invariant} condition on $(T, [\![, ,]\!])$.
\par
In \cite{Wol} a {\it quadratic} Lie triple system was originally defined as a Lie triple system $(T, [\![, ,]\!])$ along with a symmetric bilinear form $b$ satisfying (\ref{eq2.3}) and (\ref{eq2.4}). Later on, a quadratic (i.e. metrisable) Lie triple system was defined in \cite{ZSZ} as a Lie triple system satisfying the right invariant condition (\ref{eq2.3}) (see also \cite{LWD}). From \cite{K1} and \cite{K3} it follows that the conditions (\ref{eq2.3}) and (\ref{eq2.4}) are actually equivalent (that (\ref{eq2.3}) implies (\ref{eq2.4}) was first observed in \cite{K1}; see also \cite{R2} for the case of Killing forms on Lie triple systems). For completeness and because of its importance in our setting, we give a detailed proof of this result as follows.
\begin{lemma}\label{lem2.2}
 In a Lie triple system $(T, [\![, ,]\!])$ the conditions (\ref{eq2.3}) and (\ref{eq2.4}) are equivalent.
\end{lemma}
\noindent
 {\bf Proof.} First suppose (\ref{eq2.3}). We have
 \begin{align*}
 b([\![x,y,z]\!] ,u) &= b([\![z,u,x]\!], y) \; \mbox{(see (\ref{eq2.3}))} \\
 &= -b([\![u,z,x]\!] ,y) \\
 &= -b([\![x,y,u]\!], z) \; \mbox{(by (\ref{eq2.3}))} \\
 \end{align*}
so we get (\ref{eq2.4}) (see also \cite{ZSZ} and (\ref{eq3.3}) in \cite{K1}).
\par
Conversely, suppose (\ref{eq2.4}) and set $K(x,y,z,u) := b([\![x,y,z]\!] ,u)$. Then
\begin{equation}\label{eq2.5}
 K(x,y,z,u) = -K(y,x,z,u)
\end{equation}
and so, by (\ref{eq2.4}) and the symmetry of $b$,
\begin{equation}\label{eq2.6}
 K(x,y,z,u) = -K(x,y,u,z).
\end{equation}
From the other hand, (T1) implies
\begin{equation}\label{eq2.7}
 K(x,y,z,u) + K(y,z,x,u) + K(z,x,y,u) = 0.
\end{equation}
In view of (\ref{eq2.5})-(\ref{eq2.7}), Lemma 12.4 (Ch. I) in \cite{Hel} gives
\begin{equation}\label{eq2.8}
 K(x,y,z,u) = K(z,u,x,y).
\end{equation}
Therefore
\begin{align*}
 0 &= b([\![x,y,z]\!] ,u) + b(z,[\![x,y,u]\!]) \; \mbox{(see (\ref{eq2.4}))} \\
 &= b([\![z,u,x]\!] ,y) + b(z,[\![x,y,u]\!]) \; \mbox{(by (\ref{eq2.8}))} \\
 &= - b([\![u,z,x]\!] ,u) + b([\![x,y,u]\!], z)
\end{align*}
so we get (\ref{eq2.3}) (see also \cite{K3}, Proposition 1). \hfill $\Box$\\
\par
Observe that our expression of $K(x,y,z,u)$ here is different from that of \cite{K3} since we are working in Lie triple systems instead of Lie-Yamaguti algebras as in \cite{K3}. The condition (\ref{eq2.4}) is used in \cite{AB} when dealing with pseudo-Euclidean (i.e. quadratic) Lie triple systems and in \cite{I2} in the definition of a quadratic Lie-Yamaguti algebra, while (\ref{eq2.3}) is used in \cite{SZ} and \cite{K1}.
\par
A Maltsev algebra $(M,*)$ along with a symmetric bilinear form $b$ satisfying (\ref{eq2.2}) is called a {\it quadratic} Maltsev algebra. Recall that quadratic Maltsev algebras were formerly termed ``quasi-classical'' Maltsev algebras in \cite{Myu}.
\begin{lemma}\label{lem2.3}
 Let $(M,*,b)$ be a quadratic Maltsev algebra and let $(M, \langle , , \rangle, b)$ be the Lie triple system associated with $(M,*)$. Then $(M, \langle , , \rangle, b)$ is a quadratic Lie triple system.
\end{lemma}
\noindent
{\bf Proof.} See \cite{AB}. \hfill
$\Box$
\\
\par
In view of Lemma \ref{lem2.2} above, we give below a more accurate (in some sense) definition of a quadratic left Bol algebra.\\
\\
\noindent
{\bf Definition 2.5.}
A quadratic Bol algebra is a pair
$(T,b)$ consisting of a left Bol algebra $(T, *, [\![, ,]\!])$ and a nondegenerate symmetric bilinear form on $(T, *, [\![, ,]\!])$ satisfying (\ref{eq2.2}) and either of the equivalent conditions (\ref{eq2.3}) or (\ref{eq2.4}), in which case $b$ is said to be invariant on $T$.
\\
\par
An example of a symmetric bilinear form on $(T, *, [\![, ,]\!])$ that satisfies (\ref{eq2.2}) and (\ref{eq2.3}) is the one of the Killing form $\kappa$ on $(T, *, [\![, ,]\!])$ that has been defined in \cite{KZ} in a similar way as for Lie-Yamaguti algebras in \cite{K1}. Next the notion of an invariant (i.e. associative) bilinear form on a right Bol algebra is introduced (\cite{KZ}) as a generalization of $\kappa$. In \cite{KZ} the solvability and semisimplicity of the standard enveloping Lie algebra $\mathfrak{g}$ of $T$ (see \cite{Mik}) were characterized in terms of $\kappa$; moreover, the Killing radical ($K$-radical) of $T$ has been defined and it was proved that $T$ is $K$-solvable if and only if
$\mathfrak{g}$ is solvable, and $T$ is $\kappa$-semisimple if and only
$\kappa$ is nondegenerate.\\
\\
\noindent
{\bf Example 2.6.}
 Let $(M,*,b)$ be a quadratic Maltsev algebra. By Lemma \ref{lem2.3} above, the invariance of $b$ in $(M,*)$ induces the invariance of $b$ in the associated Lie triple system $(M, \langle , , \rangle)$ with $\langle , , \rangle$ defined as (\ref{eq2.1}). Therefore Remark 2.2 
 and Example 2.4
 imply that $(M,*,b)$ induces a quadratic structure on the associated Bol algebra $(M, *, [\![, ,]\!])$.
 \par
 As an application one may consider the so-called ``reductive'' Maltsev algebras over a field of characteristic zero (see \cite{Myu}). These algebras are known to be quadratic (i.e. quasi-classical) so, from the preceding discussion, their associated Bol algebras are quadratic.
 \par
 Another application is when consider a semisimple Maltsev algebra $M_0$ (\cite{R1, Eld}) along with its Killing form $\beta$. Then $\beta$ is invariant and nondegenerate (\cite{Sag}) so that $(M_0 , \beta)$ induces a quadratic Bol algebra.\\
\\
\noindent
{\bf Example 2.7.}
Let $A$ be an alternative algebra and $A^-$ its commutator algebra. Then $A^-$ is a Maltsev algebra (\cite{Mal}). If $A$ is semisimple over a field of characteristic $\neq 2$, then $A$ is quadratic by its Killing form $\kappa$ (\cite{Sch}) and a result from \cite{AB} says that $(A^- , \kappa)$ is also a quadratic Maltsev algebra. Therefore Example 
2.6 above implies that the Bol algebra associated with $A^-$ is quadratic.
\\
\par
The example and result below show that the only two-dimensional quadratic Bol algebras over
$\mathbb{R}$ are quadratic Lie triple systems.\\
\\
\noindent
{\bf Example 2.8.}
 Let $(T, [\![, ,]\!])$ be the two-dimensional Lie triple system over $\mathbb{R}$ with basis $\{e_1,e_2\}$ and multiplication table
 \par
 $[\![e_1,e_2 ,e_1]\!] = e_1$, $[\![e_1,e_2 ,e_2]\!] = -e_2$.\\
 Define on $T$ a bilinear symmetric form $b$ by setting
 \par
 $b(e_1,e_1)=0=b(e_2,e_2)$, $b(e_1,e_2) = \alpha \neq 0$.\\
 Then $(T,b)$ is quadratic.
\begin{theorem}\label{th2.4}
 There does not exist any quadratic structure on a real two-dimensional Bol algebra $(T, *, [\![, ,]\!])$ with nonzero binary operation.
\end{theorem}
\noindent
{\bf Proof.}
 We must consider algebras of type II of Theorem \ref{th2.1}. We have
 \par
 $b(e_2,e_1) = b(-e_1 * e_2, e_1) = -b(e_1, e_2 * e_1) = -b(e_1, e_2) = -b(e_2,e_1)$\\
 so that $b(e_2, e_1) = 0$, and
 \par
 $b(e_2, e_2) = b(-e_1 * e_2, e_2) = - b(e_1, e_2 * e_2) = 0$.\\
 Thus $b(e_2, T) = 0$ and we conclude that $b$ is degenerate. \hfill
 $\Box$\\

\noindent
{\bf Definition 2.9.} (\cite{Mik, SM1}).
Let $(T, *, [\![, ,]\!])$ be a Bol algebra. A linear map $\Pi : T \rightarrow T$ is called a pseudoderivation with companion
 $\chi \in T$ if, for all $x,y,z \in T$,
 \par
 $\Pi (x*y) = \Pi (x) * y + x * \Pi (y) + [\![x,y,\chi]\!] - (x*y)*\chi$,
 \par
 $\Pi ([\![x,y,z]\!]) = [\![\Pi (x),y ,z]\!] + [\![x,\Pi (y) ,z]\!] + [\![x,y,\Pi (z)]\!]$. \\
\par
In \cite{Mik} it is proved that the set $pder (T)$ of all $\Pi$ as in the definition above is a vector space and, moreover, it is a Lie algebra. The set $\widetilde{pder(T)}$ of pairs $(\Pi, \chi)$ also constitutes a Lie algebra. If consider the map $L(u,v)$, $u,v \in T$, defined by $L(u,v)(w) := [\![u,v,w]\!]$, $\forall w \in T$, then by (T2) and (T3), $L(u,v)$ is a pseudoderivation with companion $u*v$. If denote by
$\mathfrak{h}$ the set of pairs
$(L(u,v), u*v)$ for all $u,v \in T$, then $\mathfrak{h}$ is a Lie subalgebra of $\widetilde{pder(T)}$. From the other hand, the direct sum $T \dot + \widetilde{pder(T)}$ is a Lie algebra (\cite{Mik}) so that $\mathfrak{g} := T \dot + \mathfrak{h}$ is a Lie subalgebra of $T \dot + \widetilde{pder(T)}$ that is called the standard enveloping Lie algebra of $T$.
\par
As for Lie-Yamaguti algebras in \cite{K2} and for Lie triple systems in \cite{ZSZ}, it could be shown that a nondegenerate invariant symmetric bilinear form $b$ on $T$ uniquely extends to a form $\beta$ on the standard enveloping Lie algebra $\mathfrak{g}$ of $T$. Indeed define on $\mathfrak{g}$ a form $\beta$ as follows:
\par
$\beta |_T := b$,
\par
$\beta (T,\mathfrak{h}) = 0 = \beta (\mathfrak{h}, T)$,
\par
$\beta (L(u,v),L(x,y)) := b(L(u,v)(x),y)$,\\
for all $u,v,x,y \in T$ and $\beta$ extends to whole $\mathfrak{h}$ by bilinearity of $b$. We have
\begin{align*}
\beta (L(u,v),L(x,y)) &= b([\![u,v,x]\!],y)\\
&= b(u, [\![y,x,v]\!]) \; \mbox{(by (\ref{eq2.3}))}\\
&= - b(\alpha (x,y)(v),u)\\
&= - \beta (L(x,y), L(v,u))\\
&= \beta (L(x,y), L(u,v))
\end{align*}
so $\beta$ is symmetric. Next, as in \cite{K2} (proof of the Theorem) or in \cite{ZSZ} (Lemma 3.2, Lemma 3.3 and Theorem 3.4), one proves that
$\beta$ is well-defined, invariant on $\mathfrak{g}$ and $\beta$ is nondegenerate if and only if $b$ is nondegenerate.\\

\noindent
{\bf Definition 2.10.}
Let $(T,b)$ be a quadratic Bol algebra. An element $f \in End(T)$ is said to be $b$-symmetric (resp. $b$-skew-symmetric) if $b(f(x),y) = b(x,f(y))$ (resp. $b(f(x),y) =- b(x,f(y))$) for all $x,y \in T$.
\begin{proposition}\label{prop2.5}
 Let $(T,b)$ be a quadratic Bol algebra. For a $(\Pi, \chi) \in \widetilde{pder(T)}$ define the map $\phi (\Pi) : T \times T \rightarrow \mathbb{R}$ by
 \begin{equation}\label{eq2.9}
 \phi (\Pi) (x,y) = b(\Pi (x),y) + b(x,\Pi (y))
 \end{equation}
 for all $x,y \in T$. Then $\phi (\Pi)$ is a symmetric bilinear form and $\phi (\Pi)$ is invariant on $T$ (specifically $\phi (\Pi)$ satisfies {\rm (\ref{eq2.2})}) if and only if the map ${\mathcal R}_{\chi}(w) := R(w,\chi) - r(\chi)r(w)$ is $b$-skew-symmetric for all $w \in T$, where $r(x)(y) := y*x$ and $R(x,y)(z) := [\![z,x,y]\!]$ for all $x,y,z \in T$.
\end{proposition}
\noindent
{\bf Proof.}
Since $\Pi$ is linear and $b$ is symmetric and bilinear, it follows that $\phi (\Pi)$ is symmetric and bilinear.
\par
For any $x,y,z,u \in T$ we have
\begin{align*}
 \phi (\Pi)([\![x,y,z]\!],u) &= b(\Pi ([\![x,y,z]\!]),u) + b([\![x,y,z]\!], \Pi (u)) \\
 &= b([\![\Pi (x),y,z]\!],u) + b([\![x,\Pi (y),z]\!],u) \\
 &+ b([\![x,y,\Pi (z)]\!],u) - b(z,[\![x,y,\Pi (u)]\!])
\end{align*}
while
\begin{align*}
 - \phi (\Pi) (z, [\![x,y,u]\!]) &:= -b(\Pi (z), [\![x,y,u]\!]) - b(z, \Pi ([\![x,y,u]\!])) \\
 &= b(u, [\![x,y,\Pi (z)]\!]) - b(z, [\![\Pi (x),y,u]\!])\\
 &- b(z, [\![x,\Pi (y),u]\!]) - b(z, [\![x,y,\Pi (u)]\!])\\
 &\mbox{(by (\ref{eq2.4}) and the fact that $b$ is symmetric and $\Pi$ is}\\
 &\mbox{a pseudoderivation)}\\
 &=b(u, [\![x,y,\Pi (z)]\!]) + b([\![\Pi (x),y,z]\!], u)\\
 &+ b([\![x,\Pi (y),z]\!], u) - b(z, [\![x,y,\Pi (u)]\!]) \; \mbox{(by (\ref{eq2.4}))}\\
 &= b([\![\Pi (x),y,z]\!], u) + b([\![x,\Pi (y),z]\!], u)\\
 &+ b([\![x,y,\Pi (z)]\!], u)- b(z, [\![x,y,\Pi (u)]\!])
\end{align*}
and so $\phi (\Pi)$ satisfies (\ref{eq2.4}).
\par
For any $x,y,z \in T$, using (\ref{eq2.9}), (\ref{eq2.2}) (\ref{eq2.3}) and the fact that $\Pi$ is a pseudoderivation, we have
\par
$\phi (\Pi)(x*y,z) = b(\Pi (x)*y,z) + b(x*\Pi (y),z)$
\par
\hspace{3.0cm} $+ b([\![x,y,\chi]\!],z) - b(xy*\chi, z) - b(y, x*\Pi(z))$\\
and
\par
$- \phi (\Pi)(y,x*z) = b(z,x*\Pi (y)) - b(y,\Pi (x)*z)$
\par
\hspace{3.0cm} $-b(y,x*\Pi (z))- b(y,[\![x,z,\chi]\!]) + b(y,xz*\chi)$ \\
so that $\phi (\Pi)$ satisfies (\ref{eq2.2}) if and only if
\par
$b([\![x,y,\chi]\!]- xy*\chi, z) = - b(y,[\![x,z,\chi]\!] - xz*\chi)$\\
i.e.
\par
$b({\mathcal R}_{\chi}(x)(z),y) = -b({\mathcal R}_{\chi}(x)(y),z)$.\\
This completes the proof. \hfill
 $\Box$\\
\par
From the proof of the preceding result, the following corollary becomes obvious.
\begin{corollary}\label{cor2.6}
 Let $(T, [\![, ,]\!], b)$ be a quadratic Lie triple system. For a derivation $D$ of $(T,[\![, ,]\!])$, define the map $\phi (D) :T \times T \rightarrow \mathbb{R}$ by $\phi (D)(x,y) = b(D(x),y) + b(x,D(y))$ for all $x,y \in T$. Then $\phi (D)$ is a bilinear, symmetric and invariant form. \hfill $\Box$
\end{corollary}
In $\widetilde{pder(T)}$ denote by
$\widetilde{pder_a (T)}$ the subspace spanned by all $b$-skew-symmetric pseudoderivations of $(T,b)$ and by
$\mathfrak{B}_s (T)$ the vector space spanned by all symmetric bilinear forms on $T$. Then, as vector spaces, $\widetilde{pder(T)}$ is (isomorphic to) a direct sum of
$\widetilde{pder_a (T)}$ and a subspace in $\mathfrak{B}_s (T)$. Specifically, the following result holds.
\begin{proposition}\label{prop2.7}
 Let $(T,b)$ be a quadratic Bol algebra. Then the short sequence
 \par
 $0 \longrightarrow \widetilde{pder_a (T)}\stackrel{i}{\longrightarrow} \widetilde{pder (T)} \stackrel{\phi}{\longrightarrow} Im \phi \longrightarrow 0$\\
 is exact, where $i$ denotes the inclusion map and $\phi (\Pi)$ is defined by (\ref{eq2.9}) for all
 $(\Pi , \chi) \in \widetilde{pder (T)}$.
\end{proposition}
\noindent
{\bf Proof.}
A proof follows from (\ref{eq2.9}) and the definition of a $b$-skew-symmetric pseudoderivation. \hfill
$\Box$\\

\noindent
{\bf Remark 2.11.}
If $(T,b)$ reduces to a quadratic Lie triple system $(T, [\![, ,]\!], b)$, then Proposition \ref{prop2.7} implies the following exact sequence
\par
$0 \longrightarrow {Der_a (T)}\stackrel{i}{\longrightarrow} {Der (T)} \stackrel{\phi}{\longrightarrow} Im \phi \longrightarrow 0$,\\
where $\phi(D)(x,y) = b(D(x),y) + b(x,D(y))$ for all $D \in Der(T)$, so that $Der(T) \cong Der_a (T) \oplus Im \phi$. Observe that, in this case, $\phi (D)(x,y)$ is invariant by Corollary \ref{cor2.6}.

\section{Dual Representations}

\noindent
In this section, given the adjoint representation for a Bol algebra $T$, we define its dual in case when a quadratic structure is defined on $T$. For a not necessarily quadratic Bol algebra $T$, our triple of dual maps does not constitute a representation in general as defined in \cite{I1} unless some conditions are fulfilled, although it is still a representation for the underlying Lie triple system.
\par
In \cite{I1} a representation for a Bol algebra $(T, *, [\![, ,]\!])$ was defined as a quadruple $(V, \rho, \theta, D_{\theta})$, where $V$ is a vector space, $\rho : T \rightarrow End(V)$ a linear map, $\theta , D_{\theta} : T \times T \rightarrow End(V)$ bilinear maps satisfying
\par
(R1) $D_{\theta}(x,y) + \theta (x,y) - \theta (y,x) = 0$,
\par
(R21) $[D_{\theta}(x,y), \rho (z)] = \rho ([\![x,y,z]\!]) + \rho (x*y) \rho (z) - \theta (z, x*y)$,
\par
(R22) $\theta (x,y*z) = \rho (y) \theta (x,z) - \rho (z) \theta (x,y) - (D_{\theta} (y,z) - \rho (y*z)) \rho (x)$,
\par
(R31) $[D_{\theta} (u,v),D_{\theta} (x,y)] = D_{\theta} ([\![u,v,x]\!],y)+ D_{\theta} (x, [\![u,v,y]\!])$,
\par
(R32) $[D_{\theta} (u,v),\theta (x,y)] = \theta ([\![u,v,x]\!],y)+ \theta (x, [\![u,v,y]\!])$,
\par
(R33) $\theta (u, [\![x,y,z]\!]) = \theta (y,z) \theta (u,x) - \theta (x,z) \theta (u,y)+ D_{\theta}(x,y)\theta (u,z)$\\
for all $u,v,x,y,z \in T$.
\par
Because of (R1) we can write $(V,\rho , \theta)$ in place of $(V, \rho, \theta, D_{\theta})$. Observe that, if set $l(u)(v) := u*v$, $R(u,v)(w) := [\![w,u,v]\!]$ and $L(u,v)(w) := [\![u,v,w]\!]$ for all $u,v,w \in T$, then the maps $l : T \rightarrow End(T)$, $u \mapsto l(u)$, $R : T \times T \rightarrow End(T)$, $(u,v) \mapsto R(u,v)$, and $L : T \times T \rightarrow End(T)$, $(u,v) \mapsto L(u,v)$ define a representation $(T,l,R,L)$ of $(T, *, [\![, ,]\!])$ that is called the {\it adjoint representation} of $(T, *, [\![, ,]\!])$. The map $r(u)$ is defined as $r(u)(v) := v*u$, $u,v \in T$; see also Proposition \ref{prop2.5}.
\par
Now let $V^*$ be the dual space of $V$ and define a linear map ${\rho}^* : T \rightarrow End(V^*)$ and bilinear maps ${\theta}^*, D^*_{\theta} : T \times T \rightarrow End(V^*)$ by
\begin{equation}\label{eq3.1}
{\rho}^* (x)(f) := -f \rho (x),
\end{equation}
\begin{equation}\label{eq3.2}
{\theta}^* (x,y)(f) := -f \theta (x,y),
\end{equation}
\begin{equation}\label{eq3.3}
D^*_{\theta} (x,y)(f) := -f D_{\theta} (x,y).
\end{equation}
Next, as in \cite{SZ}, define the switching operator $\iota : {\otimes}^2 T \rightarrow T$ by $\iota (x \otimes y) := y \otimes x$, $\forall x \otimes y \in {\otimes}^2 T$. Then one observes that $D^*_{\theta} (x,y)= D_{-{\theta}^* \iota} (x,y)$, $\forall x,y \in T$.
\begin{lemma}\label{lem3.1}
 Let $(V, \theta, D_{\theta})$ be a representation of the Lie triple system $(T, [\![, ,]\!])$. Then
 $(V^* , -{\theta}^* \iota)$ is a representation of
 $(T, [\![, ,]\!])$.
\end{lemma}
\noindent
{\bf Proof.} A proof follows by a straightforward computation of (R1), (R31)-(R33) using the fact that $(V, \theta, D_{\theta})$ is a representation of $(T, [\![, ,]\!])$. See also section 3 in \cite{LWD} or Proposition 2.12 in \cite{SZ} in case of Lie-Yamaguti algebras. \hfill
$\Box$
\begin{theorem}\label{th3.2}
Let $(T,l,R,L)$ be the adjoint representation of a quadratic Bol algebra $(T,b)$. Then $(T^* , l^* ,-R^* \iota, L_{-R^* \iota})$ is a representation of $(T,b)$ (the dual of the adjoint representation
$(T,l,R,L)$ of $(T,b)$ that is called the coadjoint representation of
$(T,b)$).
\end{theorem}
\noindent
{\bf Proof.} That $(-R^* \iota, L_{-R^* \iota})$ satisfies (R1), (R31)-(R33) follows from Lemma \ref{lem3.1} (observe that the quadratic structure of $T$ is not needed here). It remains to check (R21) and (R22) for the triple $(l^* , -R^* \iota, L_{-R^* \iota})$.
\par
We have \\
$0= b(([L(x,y),l(z)] - l([\![x,y,z]\!])- l(x*y)l(z) + R(z,x*y))(v),u)$ (by (R21))\\
$= b([\![x,y,z*v]\!] - z*[\![x,y,v]\!] - [\![x,y,z]\!]*v - (x*y)*(z*v)+ [\![v,z,x*y]\!], u)$\\
$= -b(z*v, [\![x,y,u]\!]) + b([\![x,y,v]\!], z*u) + b(v, [\![x,y,z]\!]*u)$\\
$ - b(v*z, (x*y)*u) + b(v, [\![u,x*y,z]\!])$ (by (\ref{eq2.2}), (\ref{eq2.3}) and (\ref{eq2.4}))\\
$= b(v, z*[\![x,y,u]\!]) - b(v, [\![x,y,z*u]\!])+b(v,[\![x,y,z]\!]*u)$\\
$- b(v, z*((x*y)*u)) + b(v, [\![u,x*y,z]\!]$ (by (\ref{eq2.2}) and (\ref{eq2.4}))\\
$= b(v, l(z)L(x,y)(u) - L(x,y)l(z)(u) + l([\![x,y,z]\!])(u)$\\
$- l(z)l(x*y)(u) + R(x*y,z)(u))$\\
so the nondegeneracy of $b$ implies\\
$fl(z)L(x,y) - fL(x,y)l(z)+fl([\![x,y,z]\!])$\\
$- fl(z)l(x*y) + fR(x*y,z) = 0$,\\
$\forall x,y,z \in T$ and $f \in T^*$ i.e., by (\ref{eq3.1})-(\ref{eq3.3}),\\
$[L_{-R^* \iota} (x,y), l^*(z)]- l^* ([\![x,y,z]\!]) - l^* (x*y)l^* (z) + (-R^* {\iota})(z,x*y) = 0$.\\
Thus $(T^* , l^* ,-R^* \iota, L_{-R^* \iota})$ verifies (\ref{eq2.1}).
\par
Likewise, from the equality\\
$0= b(R(x,y*z)(v) - l(y)R(x,z)(v) + l(z)R(x,y)(v)$
\par
\hspace{0.5cm}$+ L(y,z)l(x)(v) - l(y*z)l(x)(v), u)$ (by (R22)),\\
using (\ref{eq2.2}), (\ref{eq2.3}) or (\ref{eq2.4})) whenever applicable and next the nondegeneracy of $b$, one gets\\
$fR(y*z,x) + fR(z,x)l(y) -fR(y,x)l(z)+ fl(x)L(y,z) - fl(x)l(y*z)$
$= 0$,\\
$\forall x,y,z \in T$ and $f \in T^*$ i.e.\\
$(-R^* \iota) (x,y*z) - l^* (y)(-R^* \iota) (x,z) + l^* (z)(-R^* \iota) (x,y) + L_{-R^* \iota} (y,z)l^*(x)$\\
$- l^* (y*z)l^* (x) = 0$\\
which is (R22) for $(l^* ,-R^* \iota, L_{-R^* \iota})$. This completes the proof. \hfill $\Box$\\

\noindent
{\bf Definition 3.1.}
Let $T$ be a Bol algebra, $(V_1 , {\rho}_1 , {\theta}_1)$ and $(V_2 , {\rho}_2 , {\theta}_2)$ two representations of $T$. A homomorphism from $(V_1 , {\rho}_1 , {\theta}_1)$ to $(V_2 , {\rho}_2 , {\theta}_2)$ is a linear map $\psi : V_1 \rightarrow V_2$ such that
\par
$\psi ({\rho}_1 (x)(v)) = {\rho}_2 (x) (\psi (v))$,
\par
$\psi ({\theta}_1 (x,y)(v)) = {\theta}_2 (x,y) (\psi (v))$\\
for all $x,y \in T$, $v \in V_1$. If, moreover, $\psi$ is a bijection then $(V_1 , {\rho}_1 , {\theta}_1)$ and $(V_2 , {\rho}_2 , {\theta}_2)$ are said to be isomorphic.
\begin{proposition}\label{prop3.3}
 Let $(T,b)$ be a quadratic Bol algebra. Then the adjoint representation $(T,l,R,L)$ is isomorphic to the coadjoint representation $(T^* , l^* ,-R^* \iota, L_{-R^* \iota})$ of
 $(T,b)$.
\end{proposition}
\noindent
{\bf Proof.} Let $b^{\sharp} : T \rightarrow T^*$, $x \mapsto b^{\sharp} (x)$, be the linear map defined by $b^{\sharp} (x)(y) := b(x,y)$ for all $x,y \in T$. Then,
$\forall u,z \in T$,
\par
$b^{\sharp} (l(x)(z))(u) = b^{\sharp} (x*z)(u) = -b(z*x, u)$
\par
$= -b(z, x*u) = -b^{\sharp} (z) (l(x)(u)) = l^* (x) (b^{\sharp} (z)(u))$\\
(we used (\ref{eq2.2})) and
\par
$b^{\sharp} (R(x,y)(z))(u) = b^{\sharp} ([\![z,x,y]\!])(u) = b([\![z,x,y]\!], u) = b(z, [\![u,y,x]\!])$
\par
$= b^{\sharp} (z) (R(y,x)(u)) = -R^* {\iota} (x,y) (b^{\sharp} (z)(u))$\\
(we used (\ref{eq2.3})). This completes the proof. \hfill $\Box$\\
\par
For a given Bol algebra $(T, *, [\![, ,]\!])$ (not necessarily quadratic) with representation $(V,\rho , \theta, D_{\theta})$ and maps ${\rho}^*$, ${\theta}^*$, and $D^*_{\theta}$ as defined by (\ref{eq3.1})-(\ref{eq3.3}), the quadruple $(V^* , {\rho}^* ,-{\theta}^* \iota, D_{-{\theta}^* \iota})$ is not a dual representation for $(T, *, [\![, ,]\!])$ in general (specifically,
$({\rho}^* ,-{\theta}^* \iota, D_{-{\theta}^* \iota})$ does not satisfy (R21) and (R22)), in contrast to the case of quadratic Bol algebras with respect to the adjoint representation (see Theorem \ref{th3.2} above; we wonder whether another definition of the triple $({\rho}^*, {\theta}^*,  D^*_{\theta})$ would generate a dual representation for $T$). However, up to to some additional conditions on the pair $(\rho, \theta)$, we get a dual representation for $(T, *, [\![, ,]\!])$.\\

\begin{theorem}\label{th3.4}
Let $T$ be a Bol algebra and $(V,\rho , \theta, D_{\theta})$ its representation. Then $(V^* , {\rho}^* ,-{\theta}^* \iota, D_{-{\theta}^* \iota})$ is a representation of $T$ if and only if the following conditions hold:
 \begin{equation}\label{eq3.4}
  \rho (u*v)\rho (w) + \rho (w)\rho (u*v) = \theta (u*v, w) + \theta (w, u*v),
 \end{equation}
 \begin{equation}\label{eq3.5}
  \theta (u,v)\rho (w) - \rho (w) \theta (v,u) = 0
 \end{equation}
 for all $u,v,w \in T$.
\end{theorem}
\noindent
 {\bf Proof.} By Lemma \ref{lem3.1}, (R1), (R31)-(R33) are naturally verified for
 $(-{\theta}^* \iota, D_{-{\theta}^* \iota})$. One only needs to check that (R21) and (R22) hold for
 $({\rho}^* ,-{\theta}^* \iota, D_{-{\theta}^* \iota})$ if and only if (\ref{eq3.4}) and (\ref{eq3.5}) are satisfied.
\par
 For (R21) we have, $\forall x,y,z \in T$,
 \par
 $([D_{-{\theta}^* \iota} (x,y), {\rho}^* (z)] - {\rho}^* ([\![x,y,z]\!]) - {\rho}^* (x*y) {\rho}^* (z) + (-{\theta}^* \iota)(z,x*y)) (f)$
 \par
 $= f \rho (z) D_{\theta} (x,y) - f D_{\theta} (x,y) \rho (z) + f \rho ([\![x,y,z]\!]) - f \rho (z) \rho (x*y)$
 \par
 $+ f \theta (x*y, z)$
 \par
 $= f (- [D_{\theta} (x,y),\rho (z)] + \rho ([\![x,y,z]\!]) - \rho (z) \rho (x*y) + \theta (x*y, z))$
 \par
 $= f (- [D_{\theta} (x,y),\rho (z)] + \rho ([\![x,y,z]\!]) + \rho (x*y) \rho (z) + \theta (z,x*y))$ (if and only (\ref{eq3.4}) holds)
\par
 $= 0$ (by (R21))\\
 so (R21) holds for $({\rho}^* ,-{\theta}^* \iota, D_{-{\theta}^* \iota})$.
 \par
 For (R22) we have, $\forall x,y,z \in T$,
 \par
 $({\rho}^* (y)(-{\theta}^* \iota) (x,z) - {\rho}^* (z)(-{\theta}^* \iota) (x,y) - D_{-{\theta}^* \iota}(y,z) {\rho}^* (x)$
 \par
 $+ {\rho}^* (y*z) {\rho}^* (x) - (-{\theta}^* \iota) (x,y*z))(f)$
 \par
 $= -f( \theta (z,x)\rho (y) - \theta (y,x)\rho (z) +\rho (x) D_{\theta} (y,z) - \rho (x) \rho (y*z) + \theta (y*z,x))$
 \par
 $= -f(\rho (y)\theta (x,z) - \rho (z)\theta (x,y) + D_{\theta}(z,y)\rho (x)- \rho (z*y)\rho (x) + \theta (x,z*y))$ (if and only if (\ref{eq3.4}) and (\ref{eq3.5}) hold)
\par
 $=-f(\rho (y)\theta (x,z) - \rho (z)\theta (x,y) - (D_{\theta}(y,z) - \rho (y*z))\rho (x)-\theta (x,y*z))$
 \par
 $=0$ (by (R22))\\
 so (R22) holds for $({\rho}^* ,-{\theta}^* \iota, D_{-{\theta}^* \iota})$ and the proof is completed. \hfill $\Box$

\section{$T^*$-extensions of quadratic Bol algebras}

In this section we define a $T^*$-extension of a given quadratic Bol algebra as a generalization of a $T^*$-extension of a Lie triple system (\cite{LWD}). With this notion, starting from a quadratic Bol algebra, one gets a chain of quadratic Bol algebras.
\par
Given a Bol algebra $(T, *, [\![, ,]\!])$ and its representation $(V, \rho , \theta , D_{\theta})$, the notion of a $(2,3)$-cocycle $(\nu , \omega)$ for $(T, *, [\![, ,]\!])$ with respect to $(V, \rho , \theta , D_{\theta})$ was given in \cite{I1}. In particular the trilinear map
$\omega$ satisfies
\begin{equation}\label{eq4.1}
 \omega (x,y,z) = - \omega (y,x,z),
\end{equation}
\begin{equation}\label{eq4.2}
 {\circlearrowleft}_{x,y,z} \; \; \omega (x,y,z) = 0.
\end{equation}
\begin{theorem}\label{th4.1}
Let $(T, *, [\![, ,]\!], b)$ be a quadratic Bol algebra and $(\nu , \omega)$ a $(2,3)$-cocycle for $(T, *, [\![, ,]\!])$ with respect to its coadjoint representation  $(T^* , l^* ,-R^* \iota, L_{-R^* \iota})$. Define on the vector space $T \oplus T^*$ a binary operation $*_{\nu}$ and a ternary operation $[\![, ,]\!]_{\omega}$ by
\begin{equation}\label{eq4.3}
 (x+f) *_{\nu} (y+g) = x*y + \nu (x,y) + fl(y) - gl(x),
\end{equation}
\begin{equation}\label{eq4.4}
[\![x+f,y+g ,z+h]\!]_{\omega} = [\![x, y, z]\!] + \omega (x,y,z) + fR(z,y) - gR(z,x) + hL(y,x)
\end{equation}
for all $x,y,z \in T$ and $f,g,h \in T^*$. Then $T_{\nu, \omega} := (T \oplus T^* , *_{\nu} , [\![, ,]\!]_{\omega})$ is a Bol algebra. Moreover, the form $\tilde b : (T \oplus T^*) \times (T \oplus T^*) \rightarrow \mathbb{R}$ defined by
\par
${\tilde b} |_T = b$,
\par
$\tilde b (x+f,y+g) = f(y) + g(x)$\\
for all $x,y  \in T$ and $f,g \in T^*$, is bilinear, symmetric and nondegenerate. The form $\tilde b$ is invariant if and only if
\begin{equation}\label{eq4.5}
\nu (x,y)(z) = \nu (y,z)(x),
\end{equation}
\begin{equation}\label{eq4.6}
\omega (x,y,z)(u) = \omega (u,z,y)(x).
\end{equation}
\end{theorem}
\noindent
 {\bf Proof.} Using (\ref{eq3.1})-(\ref{eq3.3}) we write (\ref{eq4.3}) and (\ref{eq4.4}) as
 \par
$(x+f) *_{\nu} (y+g) = x*y + \nu (x,y) - l^* (y) f + l^* (x) g$ \\
and
\par
$[\![x + f, y + g , z + h ]\!]_{\omega} = [\![x, y, z]\!] + \omega (x, y, z) + (-R^* \iota) (y,z)f$
\par
\hspace{4.0cm}$ - (-R^* \iota)(x,z)g + L_{-R^* \iota} (x,y)h$\\
respectively. Then, since $(T^* , l^* ,-R^* \iota, L_{-R^* \iota})$ is a representation of $(T,b)$ (see Theorem \ref{th3.2}), Theorem 3.4 in \cite{I1} implies that $T_{\nu , \omega}$ is a Bol algebra.
\par
Clearly $\tilde b$ is bilinear and symmetric. Now let $y+g \in T \oplus T^*$ be $\tilde b$-orthogonal to all $x+f \in T \oplus T^*$. Then $0 = \tilde b (x+f,y+g) = f(y) + g(x)$ which implies $f(y)=0=g(x)$, so $y=0$ and $g=0$. Thus $\tilde b$ is nondegenerate.
\par
It remains to check the invariance of $\tilde b$. We have
\par
$\tilde b ((x+f) *_{\nu} (y+g),z+h)= \nu (x,y)(z) + f(y*z) - g(x*z) + h(x*y)$\\
and
\par
$\tilde b (x+f, (y+g) *_{\nu}(z+h))= f(y*z) + \nu (y,z)(x) + g(z*x) - h(y*x)$.\\
Therefore $\tilde b$ satisfies (\ref{eq2.2}) if and only if $\nu$ satisfies (\ref{eq4.5}). Likewise,
\par
$\tilde b ([\![x + f, y + g , z + h ]\!]_{\omega}, u+k) = \omega (x,y,z)(u)$
\par
$+ f([\![u, z, y ]\!]) - g([\![u, z, x ]\!]) + h([\![y, x, u ]\!]) + k([\![x, y, z ]\!])$\\
and
\par
$\tilde b ( x+f, [\![u + k, z + h, y + g ]\!]_{\omega}) = f([\![u,z,y]\!])+ \omega (u,z,y)(x)$
\par
$+ k([\![x, y, z ]\!]) - h([\![x, y, u ]\!]) + g([\![z, u, x ]\!])$\\
so that $\tilde b$ satisfies (\ref{eq2.3}) if and only if $\omega$ verifies (\ref{eq4.6}). \hfill $\Box$\\

\noindent
{\bf Definition 4.1.}
 The quadratic Bol algebra $(T_{\nu, \omega}, \tilde b)$ constructed in Theorem \ref{th4.1} is called the $T^*$-extension of the quadratic Bol algebra $(T,b)$ by the $(2,3)$-cocycle $(\nu, \omega)$. If $(\nu, \omega) = (0,0)$ then $(T_{0,0}, \tilde b)$ is called the trivial extension of $(T,b)$.\\

\noindent
{\bf Remark 4.2.}
If $x*y =0$, $\forall x,y \in T$ and $\nu =0$, then $(T_{0, \omega}, \tilde b)$ is the $T^*$-extension of the Lie triple system $(T,[\![, ,]\!])$ by the $3$-cocycle $\omega$ as defined in \cite{LWD}. Theorem \ref{th4.1} also implies that any quadratic Bol algebra gives rise to its trivial extension with double dimension, so one gets a chain of quadratic Bol algebras.\\
\par
In the proof of Theorem \ref{th4.1} above we used the right invariant condition (\ref{eq2.3}). If one uses the left invariant condition (\ref{eq2.4}), then one must compute
\par
$\tilde b (z+y, [\![x + f, y + g , u + k ]\!]_{\omega}) = - h([\![x,y,u]\!]) - \omega (x,y,u)(z)$
\par
\hfill $+ f([\![u, z, y]\!]) - g([\![u, z, x]\!])+ k([\![x, y, z]\!])$\\
and so $\tilde b$ satisfies (\ref{eq2.4}) if and only if
\begin{equation}\label{eq4.7}
\omega (x,y,z)(u) = - \omega (x,y,u)(z).
\end{equation}
Observe that this condition (\ref{eq4.7}) is the same as in the case of Lie-Yamaguti algebras (see (3.6) in \cite{I2}).
\begin{proposition}\label{prop4.2}
 The conditions (\ref{eq4.6}) and (\ref{eq4.7}) are equivalent.
\end{proposition}
\noindent
{\bf Proof.} First suppose (\ref{eq4.6}). Then
\par
$\omega (x,y,z)(u) = \omega (u,z,y)(x) = -\omega (z,u,y)(x) = -\omega (x,y,u)(z)$ (by (\ref{eq4.6}))\\
so we get (\ref{eq4.7}).
\par
For the converse we proceed as follows. From (\ref{eq4.7}), switching $x$ and $u$, $y$ and $z$, we get
\begin{equation}\label{eq4.8}
\omega (u,z,y)(x) = - \omega (u,z,x)(y).
\end{equation}
Subtracting memberwise (\ref{eq4.8}) from (\ref{eq4.7}) we get
\par
$\omega (x,y,z)(u) - \omega (u,z,y)(x) = - \omega (x,y,u)(z) + \omega (u,z,x)(y)$
\par
$= \omega (y,u,x)(z) - \omega (u,x,z)(y) + \omega (u,z,x)(y)$ (by (\ref{eq4.2}) and (\ref{eq4.7}))
\par
$= \omega (y,u,x)(z) - \omega (u,x,z)(y) - \omega (z,x,u)(y) - \omega (x,u,z)(y)$ (by (\ref{eq4.2}))
\par
$= \omega (y,u,x)(z) - \omega (z,x,u)(y)$ (by (\ref{eq4.1}))
\par
$= \omega (u,y,z)(x) - \omega (x,z,y)(u)$ (by (\ref{eq4.1}) and (\ref{eq4.7})).\\
Thus we obtained
\begin{equation}\label{eq4.9}
 \omega (x,y,z)(u) - \omega (u,z,y)(x) = \omega (u,y,z)(x) - \omega (x,z,y)(u).
\end{equation}
From the other hand, applying (\ref{eq4.2}) to both side of (\ref{eq4.7}), we get
\par
$-\omega (y,z,x)(u) - \omega (z,x,y)(u) = \omega (y,u,x)(z) + \omega (u,x,y)(z)$\\
i.e., by (\ref{eq4.7}) and (\ref{eq4.1}),
\par
$\omega (y,z,u)(x) + \omega (x,z,y)(u) = \omega (u,y,z)(x) + \omega (u,x,y)(z)$.\\
Therefore we have
\par
$\omega (x,z,y)(u) - \omega (u,y,z)(x) = \omega (u,x,y)(z) - \omega (y,z,u)(x)$
\par
$= - \omega (x,y,u)(z) - \omega (y,u,x)(z) + \omega (z,u,y)(x) + \omega (u,y,z)(x)$ (by (\ref{eq4.2}))\par
$= - \omega (x,y,u)(z) - \omega (y,u,x)(z) + \omega (z,u,y)(x) - \omega (u,y,x)(z)$ (by (\ref{eq4.7}))\par
$= \omega (x,y,z)(u) - \omega (u,z,y)(x)$ (by (\ref{eq4.1}) and (\ref{eq4.7}))\\
and so
\begin{equation}\label{eq4.10}
-(\omega (u,y,z)(x) - \omega (x,z,y)(u)) = \omega (x,y,z)(u) - \omega (u,z,y)(x).
\end{equation}
From (\ref{eq4.9}) and (\ref{eq4.10}) we conclude that $\omega (x,y,z)(u) - \omega (u,z,y)(x) =0$ which is (\ref{eq4.6}). This completes the proof. \hfill $\Box$
\\

\noindent
{\bf Remark 4.3.} (i) The equivalence of (\ref{eq4.6}) and (\ref{eq4.7}) corroborates, in some sense, the equivalence of (\ref{eq2.3}) and (\ref{eq2.4}). Further properties of $T^*$-extensions of quadratic Bol algebras could be investigated as in \cite{LWD} for Lie triple systems.
\par
(ii) In view of the preceding discussion (see section 3) the study of $T^*$-extensions of not necessarily quadratic Bol algebras should be limited to those whose representations satisfy (\ref{eq3.4}) and (\ref{eq3.5}).

\end{document}